\documentclass[10pt]{article}
\usepackage[tbtags]{amsmath}
\usepackage{epsfig,amstext,amssymb,amsthm,latexsym,srcltx,enumerate,hyperref}
\allowdisplaybreaks[4]
\usepackage{amssymb,color}

\def\verDraft{1}
\def\verFinal{2}
\def\ver{2}
\ifx\ver\verDraft
\newcommand{\notindraft}[1]{}
\else
\newcommand{\notindraft}[1]{#1}
\fi

\ifx\ver\verDraft
\newcommand{\notinfinal}[1]{#1}
\fi
\ifx\ver\verFinal
\newcommand{\notinfinal}[1]{}
\fi

\def\ver{1}
\ifx\ver\verDraft
\definecolor{c20}{rgb}{0.,0.7,0.}
\definecolor{c30}{rgb}{0.,0.,1.}
\definecolor{c40}{rgb}{1,0.1,0.7}
\definecolor{c50}{rgb}{1,0,0}
\else
\definecolor{c20}{rgb}{0.,0,0.}
\definecolor{c30}{rgb}{0.,0.,0}
\definecolor{c40}{rgb}{0,00,00}
\definecolor{c50}{rgb}{0,0,0}

\fi
\def\cDD#1{\textcolor{c20}{#1}}
\def\cE#1{\textcolor{c50}{#1}}

\def\cDD#1{#1}

\def\cD#1{#1}
\def\cE#1{#1}

\def\cH#1{\textcolor{c30}{#1}}
\def\cH#1{#1}
\def\aH#1{\textcolor{c30}{#1}}
\def\aH#1{#1}
\def\bH#1{\textcolor{c30}{#1}}
\def\bH#1{#1}
\def\dH#1{\textcolor{c30}{#1}}
\def\dH#1{#1}
\def\hH#1{\textcolor{c30}{#1}}
\def\hH#1{#1}
\def\eH#1{\textcolor{c30}{#1}}

\def\cK#1{\textcolor{c40}{#1}}

\def\aK#1{\textcolor{c40}{#1}}
\def\aK#1{#1}
\def\bK#1{\textcolor{green}{#1}}
\def\bK#1{#1}
\def\cK#1{\textcolor{green}{#1}}
\def\cK#1{#1}
\def\eH#1{#1}

\newcommand{\PP}{\mathbb{P}}
\newcommand{\mean}[1]{\mathbb{E}\left\{ #1\right\}}

\newtheorem{theo}{Theorem}[section]
\newtheorem{sat}[theo]{Proposition}
\newtheorem{de}[theo]{Definition}
\newtheorem{lem}[theo]{Lemma}
\newtheorem{korr}[theo]{Corollary}
\newtheorem{remark}[theo]{Remark}
\newtheorem{remarks}[theo]{Remarks}

\newcommand{\netheo}[1]{{Theorem \ref{#1}}}

\newcommand{\prooftheo}[1]{ \textsc{Proof of Theorem} \ref{#1} }

\newcommand{\kb}[1]{\boldsymbol{#1}}
\newcommand{\vk}[1]{\kb{#1}}

\def\kal#1{{\cal{ #1}}}

\newcommand{\ve}{\varepsilon}

\newcommand{\E}[1]{\mathbb{E}\left\{#1\right\}}
\newcommand{\eg}[1]{\mathbb{E}\biggl\{#1\biggr\}}
\newcommand{\pk}[1]{\mathbb{P}\left(#1\right) }
\newcommand{\pb}[1]{\mathbb{P}\Bigl(#1 \Bigr )}

\newcommand{\R}{\!I\!\!R}

\newcommand{\inr}{\in \R}

\newcommand{\ldot}{,\ldots,}

\newcommand{\limit}[1]{\lim_{#1 \to   \infty}}

\newcommand{\BQN}{\begin{eqnarray}}
\newcommand{\EQN}{\end{eqnarray}}
\newcommand{\BQNY}{\begin{eqnarray*}}
\newcommand{\EQNY}{\end{eqnarray*}}

\newcommand{\BS}{\begin{sat}}
\newcommand{\ES}{\end{sat}}
\newcommand{\BT}{\begin{theo}}
\newcommand{\ET}{\end{theo}}
\newcommand{\BK}{\begin{korr}}
\newcommand{\EK}{\end{korr}}

\newcommand{\BD}{\begin{de}}
\newcommand{\ED}{\end{de}}

\newcommand{\BIT}{\begin{itemize}}
\newcommand{\EIT}{\end{itemize}}
\newcommand{\BDI}{\begin{description}}
\newcommand{\EDI}{\end{description}}

\newcommand{\BRM}{\begin{remarks}}
\newcommand{\ERM}{\end{remarks}}

\newcommand{\QED}{\hfill $\Box$}

\newcommand{\IF}{\infty}
\newcommand{\BTH}{\begin{theo}}
\newcommand{\ETH}{\end{theo}}
\newcommand{\BPR}{\begin{sat}}
\newcommand{\EPR}{\end{sat}}

\newcommand{\BC}{\begin{cases}}
\newcommand{\EC}{\end{cases}}
\newcommand{\COM}[1]{}
\newcommand{\BL}{\begin{lem}}
\newcommand{\EL}{\end{lem}}

\def\SI{\Sigma}

\def\Y{\vk{Y}}

\def\U{\vk{U}}

\def\1d{\{1 \ldot d\}}

\def\Y{\vk{Y}}
\newcommand{\equaldis}{\stackrel{d}{=}}

\def\Xiu{X_i(u)}
\def\Ziu{\cE{X_i(u)}}
\def\Zju{\cE{X_j(u)}}
\def\Zku{\cE{X_k(u)}}

\def\SSU{\cE{S(u)}}

\def\ftj{\cE{f}}

\def\lui{\left(\frac u{\lambda_i}\right)}
\def\luiu{\left(\frac u{\lambda}\right)}
\def\tH#1{\textcolor{c30}{#1}}
\def\tH#1{#1}

\def\zH#1{\textcolor{c30}{#1}}
\def\zD#1{\textcolor{c20}{#1}}
\def\zH#1{#1}
\def\zD#1{#1}

\begin{document}

\begin{center}
\thispagestyle{empty}

{\Large
\cH{Efficient simulation of tail probabilities for sums of log-elliptical risks}}

       \vskip 0.4 cm

         \centerline{\large
         Dominik Kortschak\footnote{
Department of Actuarial Science, Faculty of Business and Economics,
University of Lausanne, B\^{a}timent Extranef, UNIL-Dorigny, 1015 Lausanne, Switzerland, email: Dominik.Kortschak@unil.ch}
 and Enkelejd Hashorva\footnote{
Department of Actuarial Science, Faculty of Business and Economics,
University of Lausanne, B\^{a}timent Extranef, UNIL-Dorigny, 1015 Lausanne, Switzerland, email: Enkelejd.Hashorva@unil.ch}
}

\today{}

\end{center}

{\bf Abstract:}
\dH{In the framework of dependent risks it is a crucial task for risk management purposes to quantify the probability that the aggregated risk
exceeds some large value $u$. Motivated by Asmussen et al.\ (2011) in this paper we introduce a modified Asmussen-Kroese estimator for simulation of the rare event that the aggregated risk \hH{exceeds} $u$. We show that in the framework of log-Gaussian risks our novel estimator has the best possible performance i.e., it has  asymptotically vanishing relative error. For the more general class of log-elliptical risks with marginal distributions in the Gumbel max-domain of attraction we propose a modified Rojas-Nandayapa estimator of the rare events of interest, which for specific importance sampling densities has a good logarithmic performance. Our numerical results presented in this paper demonstrate the excellent performance of our novel Asmussen-Kroese \hH{algorithm}. \\
{\it Key words and phrases}: Asmussen-Kroese estimator; Rojas-Nandayapa estimator; 
log-elliptical distribution; log-Gaussian distribution; 
asymptotically vanishing relative error.
}

\section{Introduction}
Efficient simulation of the tails of aggregated dependent risks has been the topic of many recent research papers, culminating in the  \tH{contribution} Asmussen et al.\ (2011). The fact that risks -- here a synonym for random variables -- are considered to be dependent,
 poses considerable difficulties in understanding the tail behavior of the aggregated risk. Nevertheless in diverse applications from finance and insurance (Goovaerts et al.\ (2005), Valdez et al.\ (2009), Asmussen et al.\ (2011)), risk management
 (Vanduffel et al.\ (2008), Mitra and  Resncik  (2009)),  wireless communications (Pratesi et al.\ (2006), Tellambura (2008)) a few to be mentioned here,  correlated log-Gaussian (log-normal) risks appear naturally.\\
In this paper we will allow that the \tH{parameters} of the log-normal distribution depend on $u$. Therefore
let $\vk{N}=(N_1,\ldots,N_d)^\top$ be a vector of $d$ independent standard Gaussian random variables, and let $A_u,u>0$
be a lower non-singular triangular matrix. For $\SI_u\tH{=}A_u(A_u)^\top$ assume that $\SI_u$ is a correlation matrix, i.e.,
\[\sigma_{11}(u)=\cdots=\sigma_{dd}(u)=1, \quad \sigma_{ij}(u)\in [-1,1], \quad i\not=j,\quad u>0.\]
Set for $u>0$
\BQN \label{eq:r:yu}
(Y_1(u),\ldots,Y_d(u))^\top=A_u\vk{N}
\EQN
and  define
$$S(u)= \sum_{i=1}^d X_i(u), \quad \text{\hH{with}}\quad X_i(u)=\lambda_i e^{\beta_i \gamma_u Y_i(u) }, \quad i\le d ,$$
where $\lambda_i,\beta_i, \gamma_u,u>0$ are given positive constants. In this paper we are interested in the numerical estimation of
\[
 \alpha(u)=\pk{S(u)> u}.
\]
For $d=2$ and \tH{both} $A_u, \gamma_u$ \tH{being constant with respect to $u$, the asymptotic expansion}
\BQN \label{AA06}
\alpha(u)\sim \PP(X_1(u)>u)+\PP(X_2\tH{(u)}>u), \quad u\to \IF
\EQN
has been first \tH{derived} in Asmussen and Rojas-Nandayapa (2008) (see for a \tH{heuristical} derivation Albrecher et al.\ (2006)). Similar \tH{asymptotic results to \eqref{AA06} for general $\gamma_u, A_u$} have been derived in Asmussen et al. (2011) and Hashorva (2013a).
\tH{In our notation $f(u) \sim g(u)$ means that $\limit{u} f(u)/g(u)=1$ for $f,g$
two given function.}\\
\tH{In the light of known} numerical examples (see e.g. Mitra and Resnik (2009))
\tH{the  asymptotic expansion of $\alpha(u)$  given in \eqref{AA06}} is too crude to be useful in practice. Hence one seeks for numerical solutions for $\alpha(u)$. A widely used numerical method for this kind of problems is Monte Carlo simulation. Since $\alpha(u)\to 0$ as $u\to\infty$ we are in the classical situation of rare event simulations.
By definition, see e.g., Asmussen and Glynn (2007), an unbiased estimator $Z(u)$ of $\alpha(u)$  (i.e., a family of random variables satisfying $\mean{Z(u)}=\alpha(u)$) is said
to be (asymptotically as $u\to \IF$) logarithmically efficient if
$$
\lim_{u\to \IF} \frac{ \log \E{ Z(u)^2} }{\log \alpha(u)}=2.
$$
A concept that goes beyond that is introduced in Junea (2007), namely $Z(u)$ has
asymptotically vanishing relative error if \tH{further}
\BQN
\lim_{u \to \IF } \eg{ \biggl(\frac{Z(u)}{\alpha(u)}\biggr)^2}&=&1.
\EQN
Such estimator of $\alpha(u)$ reaches the best possible asymptotic performance.\\
It is well-known (see e.g., Cambanis et al.\ (1981)) that the $d$-dimensional standard Gaussian random vector $\vk{N}$ has the stochastic representation
$$ \vk{N} \equaldis R \U,$$
with $R>0$ such that $R^2$ is chi-square distributed with $d$ degrees of freedom being further independent of the random vector $\U$ which is uniformly distributed on the unit sphere of $\R^d$ (hereafter  $\U$ will be reserved only for such random vectors).
If we drop the distributional assumption on $R$, supposing only that it has some distribution function $F$,  
\hH{then $\Y(u),u>0$ with stochastic representation}
$$\Y(u)\equaldis  \exp( A_u R \U)$$
is a log-elliptical random vector; see Cambanis et al.\ (1981) for the basic distributional properties of elliptical random vectors.
The framework of multivariate log-elliptical risks is useful in finance and insurance models (see  e.g., Hamada and Valdez (2008),  Valdez et al.\ (2009)). A key advantage when working with elliptical and log-elliptical risks is that in our model there is no distributional restriction on each individual risk; we impose only asymptotic constrains which are \bH{satisfied}  by a large class of possible marginal distributions. Rojas-Nandayapa (2008) provided an estimator that also works for this class of distributions.\\

Organisation of the paper: \tH{In the following we review some key results from the literature. Section 3 gives details of }
\aH{our} novel Asmussen-Kroese estimator of $\alpha(u)$ \tH{which has excellent performance for} log-Gaussian risks. \aH{In Section 4 we shall introduce the modified Rojas-Nandayapa estimator which can be utilised for log-elliptical risks.
The numerical \eH{illustrations} \bH{presented} in Section 5 show the excellent performance of our modified Asmussen-Kroese estimator.
The proofs of all results are relegated to Section 6,
which is followed by an Appendix.}

\section{Details for known estimators}\label{section:formerresults}
When \tH{$X_1(u),\ldots,X_d(u)$} are \cDD{independent} random variables with common distribution function $F$ an estimator $Z_{AK}(u)$ of $\alpha(u)$
\tH{(referred to as Asmussen-Kroese estimator)} is introduced in Asmussen and Kroese (2005). Namely,  we have \tH{(set $\overline{F}= 1- F$)}
\[
 Z_{AK}\tH{(u)}= d\cDD{\cdot} \overline F\left(\max\left(u+ X_d(u)-\tH{S(u)},\max_{1 \le i< d} X_i\tH{(u)}\right)\right),
\]
which is \tH{motivated} by the \tH{following} decomposition
\BQN\label{psij}
\alpha(u)= \sum_{j=1}^d \Psi_j(u), \quad \text{with} \quad \Psi_j(u)= \pk{S(u)> u, X_j(u) =M(u)}\quad \text{and}\quad M(u)=\max_{1\le i \le d} X_j(u).
\EQN
Accounting for the dependence of the risks,
\COM{
It was natural to use this estimator also for dependent random variables; a first try  leads to the estimator
\[
 Z_{DAK}(u)=\sum_{i=1}^d \PP\left(\tH{S(u)}>u, \tH{M(u)} =X_i(u) \big|X_j(u),i\not=j\right).
\]
This estimator was introduced in  Asmussen and Rojas-Nandayapa (2006), but has \tH{been} proven to have a rather disappointing performance in the dependent case. \\
}
in the setup of log-Gaussian risks, Asmussen et al. (2011)  introduces \tH{three different} estimators of $\alpha(u)$. The first \tH{one denoted by} $Z_{IS}(u)$ is an importance sampling estimator where the importance sampling \tH{distribution} is log-Gaussian but the matrix $\Sigma$ is multiplied by some constant $\gamma_{u}$, which is deduced from an asymptotic argument. Related to this estimator is  $Z_{IS-CE}(u)$ where again the importance sampling \tH{distribution} is   log-Gaussian but this time also the mean vector can be different.
The parameters are then chosen with the cross entropy method.\\
\tH{The third estimator of $\alpha(u)$ introduced in the aforementioned paper has a vanishing relative error}. Write   $\PP(S(u)>u)$ as
$$\PP\left(S(u)>u,\max_{i\le d} X_i(u)>u\right)+\PP\left(S(u)>u,\max_{i\le d} X_i(u)\le u\right):=\alpha_1(u)+\alpha_2(u).$$
For the first term $\alpha_1(u)$ an importance sampling estimator that has vanishing relative error is suggested therein, whereas  for the second term $\alpha_2(u)$ an importance sampling estimator equivalent to $Z_{IS}(u)$ respectively $Z_{IS-CE}(u)$ is employed.
The sum of these estimators is denoted by $Z_{ISVE}(u)$ and  $Z_{ISVE-CE}(u)$, respectively.

\tH{The more general case of log-elliptical risks is addressed in  Rojas-Nandayapa (2008).}
The main idea of Rojas-Nandayapa estimator of $\alpha(u)$ is that for a log-elliptical random vector we have
$S(u)=h(R, A_u, \vk{U})$ for some function $h$, where  $R$ and $\U$ are independent. Thus conditioning on $\vk{U}$ yields
$$ \alpha(u)=\pk{S(u)> u}= \pk{h(R, A_u, \vk{U})> u}=\E{ \pk{ h(R, A_u, \vk{U})> u)\lvert \vk{U}}}, \quad u>0.$$
Denote in the following by $\vk{u}$ a simulated value (outcome) of $\vk{U}$. Since $\SI_u$ is \tH{assumed to be} positive definite,
for any fixed $u$, the equation $h(R,A_u,\vk{u})=u$ solved for $r>0$ has at most two solutions denoted by $\psi_L(u, \vk{u})$ \aH{and} $\psi_U(u, \vk{u})$.\\
For a given outcome $\vk{u}$ the function $h$ can be $S1)$ strictly decreasing, $S2)$ decreasing or increasing, and $S3)$
 strictly increasing. Both properties $S1,S2, S3$ are examined in Rojas-Nandayapa  (2008), p.\ 62. We define $\psi_L(u,\vk{u}), \psi_U(u, \vk{u})$
 as therein, for instance if $S2$ holds, then there  \aH{exist} at most two different solutions 
 \aH{satisfying}
$$\limit{u} \psi_L(u, \vk{u})=-\IF, \text{  and } \limit{u} \psi_U(u, \vk{u})=\IF.$$
The Rojas-Nandayapa  estimator of \aH{$\alpha(u)$}
is defined as
\BQN \label{Rojas}
Z_R(u) &=& \pk{R < \psi_L(u, \vk{U})}\aH{\mathbb{I}}_{\{ \psi_L(u,\vk{U})>0\}}+ \pk{R> \psi_U(u,\vk{U})}, \quad u>0.
\EQN
 \tH{Summarising}, the algorithm proposed in \aH{Rojas-Nandayapa  (2008) consists of the} following steps:\\
A. Simulate the random vector $\vk{U}$ which is uniformly distributed \bH{on the unit sphere of $\R^d$}.\\
B. Calculate $\psi_L(u, \vk{U}), \psi_U(u, \vk{U})$.\\
C. Return $\aH{Z_R}(u)$ as in \eqref{Rojas}.\\
\tH{As shown in  the aforementioned paper $Z_R(u)$} is unbiased and {\hH{logarithmically} (asymptotic) efficient} under certain restrictions on the random radius $R$.

\section{A novel Asmussen-Kroese estimator}\label{section:AK}
One reason that Asmussen-Kroese estimator has a good asymptotic behavior in the independent case is that heuristically when the sum is large then one element is large and \tH{all the others} behave in a normal way.
In this section we want to present a new modification of Asmussen-Kroese estimator that is better suited for log-Gaussian risks.
In this paper, for the efficient estimation of the tail probability $\alpha(u)= \pk{S(u)>u}$ for $u$ large we use the decomposition \eqref{psij}.
  We shall consider  the estimation, for each index $\bH{j}\le d$, of the \aH{partial} max-sum probability $\Psi_{\bH{j}}(u)$ \bH{defined in \eqref{psij}}.
  In order to compensate for the role of different components being maximal, (corresponding to different  indexes $\bH{j}$)
 we shall utilise a stratification idea. Specifically, when $\pk{\kal{I}=i}=\frac{\pk{\Ziu >u}}{\sum_{j=1}^d  \pk{\Zju >u}}, i\le d$ and $\kal{I}$ \hH{is} a random variable, then
$$
\bH{\alpha(u)=} \pk{S(u)>u}=\left({\sum_{i=1}^d  \pk{\Ziu>u}}\right) \sum_{i=1}^d \pk{\kal{I}=i} \frac{\Psi_i(u)}
{ \pk{X_i\aH{(u)}>u}},
$$
\bH{which leads to \tH{our} novel modified Asmussen-Kroese estimator of $\bH{\alpha(u)}$}
\begin{equation}
Z_{MAK}(u)=\left({\sum_{i=1}^d  \pk{\Ziu>u}}\right) \sum_{i=1}^d \mathbb{I}_{\{\kal{I}=i\}}  \frac{Z
_i(u)}{ \pk{X_i\aH{(u)}>u}},
\label{Strat}
\end{equation}
where $\mathbb{I}_{\{\cdot \}}$ is the indicator random variable \hH{and $Z_i(u)$ is \cDD{our modified Asmussen-Kroese}  estimator of $\Psi_i(u)$}. In view of Lemma   \ref{theorem:Stratification} in Appendix,
it is enough  \cDD{to show that $Z_i$ is an efficient} estimator for $\Psi_i(u)$.

Our novel Asmussen-Kroese estimator of the \aH{partial} max-sum probability $\Psi_j(u)$ is constructed by modifying the
classical Asmussen-Kroese estimator (see e.g., \cite{AsmRoj}). \cDD{We will assume that $A_u$ is chosen such that $Y_j=N_j$}. Essentially,
instead of condition\bH{ing} on $X_i(u), \aH{i\le d, i\not=j}$  like for Asmussen-Kroese estimator
we condition on $\aH{N_i,i\le d, i\not=j}$, \cH{which leads to the following} estimator \aH{(set $M(u)= \max_{i \le \bH{d}} \Xiu$)}
\bK{\begin{align*}
&\cK{Z_{j}}(u)=
\pk{ \aH{S(u)>u}, X_j(u) \aH{=} \aH{M(u)} 
\Bigl \lvert  \tH{\vk{N}}_{-j}} \\
 &=\pk{\left.\sum_{i=1}^d \lambda_i e^{\beta_i\gamma_u a_{ij} (u) N_j+\sum_{k \not = j} \beta_i \gamma_u a_{ik}\cH{(u)}N_k }>u,\aH{\lambda_j}
e^{\beta_{\aH{j}}\gamma_u a_{\aH{jj}}\cH{(u)}N_{\aH{j}}} = \max_{i \le d} \lambda_i  e^{\beta_i\gamma_u a_{ij} (u) N_j+\sum_{k \not = j} \beta_i \gamma_u a_{ik}\cH{(u)}N_k } \right|\hH{\vk{N}}_{-j}},
\end{align*}}
where $a_{ij}(u)$ is the $ij$th entry of the matrix $A_u$ and $\hH{\vk{N}}_{-j}=(N_1,\ldots,N_{j-1},N_{j+1},\ldots,N_d)$. 
Throughout in the sequel $\gamma_u,u>0$ are constants satisfying  $\limit{u} \gamma_u=\gamma\in (0,\IF)$ and
$\beta_i, \lambda_i$ are positive constants. \aH{For  $e(x),x\inr$ some function (to be specialized later) we define}
 \BQN\label{eq:eyj}
 e_i^*(u)&=&
  \beta_i\gamma_u u e\left( \lui^ {\frac 1 {\beta_i\gamma_u}}\right)\lui^ {-\frac 1 {\beta_i\gamma_u}}.
\EQN
\aK{The main result of this section is the next theorem which establishes the asymptotic properties of  $Z_{MAK}(u)$.}

\begin{theo} \label{theo:dom:AK} Define $J=\{j: \beta_{\aH{j}}\aH{=}\max_{i \le d} \beta_i\}$ and set $e(x)= x^{-1} \hH{\log (x)},x>0$.
If further for all $c>0$ and $\epsilon>0$ there exists $u_0>0$ such that for all $u>u_0$ and $\aH{i \not=j\in J }$
\BQN \label{beta:ey}
\sigma_{\aH{ij}}(u)+c \sqrt{\frac{1-\sigma_{\aH{ij}}(u)^2} {\log(u)}}
&\le & \cH{\frac{\beta_{\aH{j}}}{\beta_i}} \frac {\log(\epsilon e^*_{\cH{i}}(u))}{\log(u)},
\EQN
then the modified Asmussen-Kroese estimator $Z_{MAK}(u)$ of \bH{$\alpha(u)$} has \cH{asymptotically} vanishing relative error.
\end{theo}

\COM{
On the other hand
 if $\beta_j < \max_{ i\le d, i\not=j } \beta_i$, then for all $u$ large
$$
Z_{1}(u) \le \pk{\lambda_1  e^{\beta_1\gamma_u N_1}>\frac ud }
$$
implying
$$
\limsup_{u\to\infty}\frac{ \left(\mean{Z_{i}(u)^2}\right)}{ \Psi_1(u) \pk{S(u)>u}}=0,
$$
and hence it follows  from Lemma  \ref{theorem:Stratification} $Z_{MAK}(u)$ has \aH{asymptotically} vanishing relative error, i.e., $$\limit{u}  \frac{\E{Z^2_{MAK}(u)}}{  \alpha(u)}=0.$$
}


\begin{remark}
a) Condition \eqref{beta:ey} is forced only \eH{when} $\liminf_{u\to \IF} \sigma_{ij}(u)=1$, since
when $\sigma_{ij}(u) \le \rho< 1$ for all $u$ large \eqref{beta:ey} is satisfied for any $\aH{c}>0$.\\
\COM{ In the other case
$\limit{u}\sigma_{1i}(u)=1$, then
$$ \sigma_{1i}(u)+c \sqrt{\frac{1-\sigma_{1i}(u)^2} {\log(u)}} = 1 - \sqrt{1- \sigma_{1i}(u)} \Bigl[ 1+ c \sqrt{2}(1+o(1))
 \sqrt{ \frac{1- \sigma_{1i}(u)}{\log u}}\Bigr], \quad u\to \infty.$$
}
b)
\eH{In order} to evaluate $Z_{\bH{j}}(u)$ we have to modify the matrix $A_u$ in such a way that $Y_{j}=N_j$  which means that for every $j$ we have to compute a Cholesky factorization of a matrix.
Further  we need \tH{to determine $x$} satisfying the equation $\sum_{i=1}^d c_i e^{d_i x}=u$.
As shown in Rojas-Nandayapa (2008) 
\aH{such an} $x$ can be quite efficiently found by Newton\aH{'s} method.\\
c) The recent paper Kortschak (2011) derives second-order asymptotic results for dependent risks with regularly varying tails.
Similar results for our framework where risks have distributions in the Gumbel MDA (and therefore have no regularly varying tails),
 will be derived in a forthcoming manuscript.
\end{remark}

\section{The modified \cH{Rojas-Nandayapa} estimator}

An key result of this section is Theorem \ref{mytheo3} below, which motivates a modification of
the  algorithm of Rojas-Nandayapa (2008). Our novel modified Rojas-Nandayapa estimator introduced in \eqref{RN} is
\aH{\hH{logarithmically} efficient, and moreover behaves asymptotically significantly better than the original one.} Specifically, \aH{our algorithm is constructed under the following} modifications:
\def\vcT{\bH{\vk{\Theta}}}
\begin{itemize}
 \item[i)] As for  Asmussen-Kroese estimator we condition on \hH{the element which is} the maximum.
 \item[ii)] We use importance sampling on $ \vcT :=A_u \vk{U}$.
 \item[iii)] We \hH{employ} the same stratification method as in Eq.\ \eqref{Strat}.
\end{itemize}
\tH{We note in passing that  Rojas-Nandayapa (2008) considers only the case that $A_u$ is constant in $u$.}\\
For a given index $j$ assume that $A_u$ is chosen in such a way that $\bH{\Theta_j} =(A_u \vk{U})_j=\bH{U}_j$. We will only change the distribution of $\bH{\Theta}_j$
\cDD{which possesses the}  probability density function (pdf)
\BQN\label{dtheta}
f(\aH{\theta})&=&
\frac{\Gamma(d/2)}{\sqrt{\pi} \Gamma((d-1)/2)} (1-\aH{\theta}^2)^{\frac{d-3}2}, \quad \aH{\theta}\in (-1,1),
\EQN
\tH{where   $\Gamma(\cdot)$ is the Euler gamma function. We write  $f_{IS}$ for the corresponding pdf of $\Theta_j$} under the importance sampling measure. We then use the estimator
$$
\hat Z_j(u)=\pb{S(u)>u,X_j(u)=\aH{ M(u)} 
\Bigl \lvert   \vcT } \frac{f(\aH{\Theta_j})}{f_{IS}(\aH{ \Theta_j})}
$$
to estimate $\Psi_j(u)$ and
\BQN \label{RN}
Z_{RN}(u)=\left({\sum_{i=1}^d  \pk{\Ziu>u}}\right) \sum_{i=1}^d \mathbb{I}_{\{\kal{I}=i\}}  \frac{\hat Z
_i(u)}{ \pk{X_i>u}}
\EQN
\eH{as an estimator for $ \alpha(u)=\pk{S(u)>u}$}. As in Section \ref{section:AK} we only have to show that the estimators $\hat Z_j(u)$ are asymptotically efficient.

For our investigations we  shall assume that the distribution function $F$ of $R$ with infinite upper endpoint,
belongs to the Gumbel \aH{MDA} with some positive scaling function $\nu$, i.e.,
\BQN \label{eq:rdfd:2}
\limit{u}
\frac{1- F(u+x \bH{\nu}( u))} {1- F(u)} &=& \exp(-x),\quad \forall x\inr,
\EQN
which we abbreviated hereafter as  $F \in GMDA(\nu)$ or $R \in GMDA(\nu)$. We suppose in the following that
\BQN \label{eq:main:e}
\limit{u}  \bH{\nu}( u) = 0,\quad \text{and}\quad \limit{u} u   \dH{\nu}( \log(u)) =\infty.
\EQN

  \BT \label{mytheo3} Suppose that \eqref{eq:main:e} holds.
  \aH{If further, for  $j$  with $\beta_j=\max_{1\le i\le d}\beta_i$  \bH{condition \eqref{beta:ey} is satisfied} for any $i\not=j, i\le d$},
  then we have
  \BQN\label{mytheo3:a}
  \PP\left(\SSU>u\right)&\sim & \sum_{i=1}^d \PP\left( \Ziu>u\right).
  \EQN
  \ET
\begin{remark}\label{remark:conditionrho}
a) The sum in \eqref{mytheo3:a} can be   reduced to the sum over the indices $i$ \hH{such that} $\beta_i=\max_{1\le j\le d}\beta_j$ and $\lambda_i=\max_{j:\beta_j=\beta_i} \lambda_j$.\\
b) The scaling function $ \bH{\nu}( \cdot)$ is asymptotically equivalent to the mean excess function $\E{(R- x) \lvert R>x}$.\\
If $\beta_i=\beta_j$ and $\lim_{u\to\infty} \log(e_j^*(u))/\log(u)=1$, then \eqref{beta:ey} is for example fulfilled when ($\sigma_{ij}(u)<1$)
$$
\limsup_{ u\to \infty} \frac{-\log\left(\frac{e^*_j(u)}{u}\right) }{(1-\sigma_{ij}(u))\log(u)} <1.
$$
Note that  $e_j^*(\cdot)$ above is defined in \eqref{eq:eyj} where  $e(u)= u  \bH{\nu}( \log(u))$.
\end{remark}

We shall consider in the following importance sampling pdf $f_{IS}$ given by
\BQN
f_{IS}(a,b,x)=2^{-(a+b-1)}\frac{\Gamma(a+b)}{\Gamma(a)\Gamma(b)}(1+x)^{a-1}(1-x)^{b-1}, \quad x\in [-1,1],
\EQN
with $a,b$ positive constants. For our numerical results we choose \bH{the constant} $a$ to be large, say equal to \bK{10}.
Next set
\BQN\label{ddm}
beta=\max_{i\le d} \beta_i, \quad \lambda=\max_{i:\beta_i=\beta} \lambda_i, \quad d_m=|\{i: \beta_i=\hH{\beta}, \lambda_i=\lambda\}|.
\EQN
Whenever \dH{the index} $i$ is such that  $\beta_i=\hH{\beta}$ and $\lambda_i=\lambda$ we define $e^*(u):=\bH{e_i^*}(u)$.

\begin{theo}\label{lemma:rojas1}
Let the  assumptions of Theorem \ref{mytheo3} be fulfilled. Further assume that the function $e(u)=\bH{u \nu(\log u)}$ is of bounded variation i.e., for all $c>0$
\begin{equation}\label{bed:rojas}
\limsup_{u\to\infty}\frac{e(c u)}{e(u)}<\infty.
\end{equation}
If the importance sampling pdf $f_{IS}$ has parameters $a>0$ and $b=\beta(u)=\bH{\log(u)}/\log(\bH{u/\dH{e^*}(u)}),$ then
\BQN
\frac{\mean{\hat Z_{RN}^2(u)}}{ \PP\left(  S(u)>u \right)^2} \sim    \frac{e \Gamma(d/2) } {2\sqrt{\pi}\Gamma((d-1)/2) }\log \left( \frac{u\log(u)}{\bH{e^*(u)}} \right).
\EQN
\end{theo}

 \begin{remark}\label{forappendix} a) In the \eH{log-Gaussian} case it follows that the standard error  $\sqrt{\frac{\mean{\hat Z_{RN}^2(u)}}{ \PP\left(  S(u)>u \right)^2}}$ is of order $\sqrt{\log(\log(u))}$ and hence it remains small \bK{even for  relatively} large values  $u$.\\
b) For the original Rojas-Nandayapa estimator \eH{$Z_R(u)$ defined in (3.8)} we \eH{obtain} under the same conditions  as in \netheo{lemma:rojas1} that (recall $d_m$ is defined in \eqref{ddm})
 $$
\frac{ \mean{Z_{R}^2(u)}}{ \PP\left(  S(u)>u \right)^2}    \gtrsim \frac 1{d_m} \frac{2\sqrt{\pi}} {\Gamma(d/2)} \left(\frac{u\log(u)}{\bH{e^*(u)}}\right)^{\bH{\frac{d-1} 2}}.
$$
\aK{In the log-\dH{Gaussian} case it follows that the standard error   is of order $\log(u)^{\frac{d-1} 4}$ and hence significantly bigger \hH{than} for the modified estimator.}
 \end{remark}

\section{Numerical examples}\label{section:numerics}
In this section we present some examples on rare-event estimation. In order to compare our results, we \zH{refer} to the examples of Asmussen et al.\ (2011). \zH{Specifically}, we consider the
case of a multivariate \dH{log-Gaussian} distribution with $d=10$,
$$\mu_i=i-10, \quad \sigma_{ii}^2=i, \quad \zH{i \le d}$$
 and
 $$\rho_{ij}\in\{0,0.4,0.9\}, \quad u\in \{20000,40000,500000\}.$$
\zH{In order to obtain} reliable estimates for the variance we \dH{performed}  $10^7$ simulations for each \zH{proposed} estimator.
Beside the standard error  ($\sqrt{\mathbb V\text{ar}\{Z(u)\}}$) and the coefficient of variation  $\sqrt{\mathbb V\text{ar}Z(u)}/\mean{Z(u)}$ we also provide the needed time for the evaluation (for $5*10^5$ simulations since this is the number of simulations used in Asmussen et al.\ (2011), computations where \dH{carried out} in R \cite{R}) and the Efficiency defined by
\[
 \frac{V\text{ar} \{\text{CMC-estimator}\}\times \text{Computation-time}\{\text{CMC-estimator}\}}{V\text{ar} \{\text{Estimator}\}\times \text{Computation-time}\{\text{Estimator}\}} .
\]

We compare our estimators to the Crude Monte Carlo estimator $Z_{CMC}=\mathbb{I}_{\{S(u)>u\}}$ and the importance sampling
estimators defined in the aforementioned paper (compare Section \ref{section:formerresults}).\\

\bigskip
\underline{\bf $\rho=0$}:
\cDD{ In this case the by far best  estimator is the novel modified Asmussen-Kroese estimator (MAK) that corresponds in this case to the classical Asmussen-Kroese estimator;} in Table \ref{tab02} for example it outperforms
the other estimators by a factor of $100$. Further the performance of the modified Rojas-Nandayapa estimator lies between the one of
the IS respectively IS-CE and ISVE respectively ISVE-CE.\\
 Comparing our implementation with \zD{that} of Asmussen et al.\ (2011) we see that
our implementation of the estimators IS-CE, ISVE and ISVE-CE is considerably slower. \zH{However},
for other values of $\rho$ our implementation of  ISVE and ISVE-CE is more efficient. Perhaps in Asmussen et al.\ (2011) a
different implementation for the independent case was used. For IS-CE we do not have a plausible explanation why our estimator is slower,
but it only shows that the used implementation of an estimator can be important for the comparison with other estimators.
If we compare the standard errors we see that there can be considerable differences. Here one should note that the standard error is
only estimated and hence can only be estimated with a certain amount of uncertainty. Since we used considerably more simulations than in
Asmussen et al. (2011) we will assume that our results are more accurate.\\
\zH{In order to} get an idea for the uncertainty involved,  one can consider  Table
\ref{tab03} and the results for estimator IS-CE. We see that although the reported standard error is small the error of the
estimation is relatively large, which suggest that the distribution of IS-CE is rather skewed. \zH{Therefore}, one should mistrust the
standard error for this particular estimator.\\
 Since the comparison \zH{of our findings with those in}  Asmussen et al.
(2011) for the other values of $\rho$ is similar as for $\rho=0$ we will concentrate next on our numerical results.\\

\underline{\bf $\rho=0.4$}: In this case we see that our modified Rojas-Nandayapa (RN) estimator has standard error that is comparable to the one of ISVE-CE which is the best of the estimators in Asmussen et. al. (2011). However the RN estimator suffers from a long computation time and hence in practice the ISVE-CE estimator is still preferable. On the other hand we see that our \cDD{MAK estimator} is by far the best in terms of standard variation as well as in terms of efficiency.  We have a speed up to a factor $8$ for $u=20000$ to a factor of $33$ for $u=500000$. \\

\underline{\bf $\rho=0.9$}: We observe that all estimators decrease there performance. Our RN estimator has standard error that is comparable to the one of ISVE-CE which is the best estimator in Asmussen et. al. (2011). As explained above, RN estimator suffers from a long computation time.
  Similarly, our \cDD{MAK estimator} is by far the best in terms of standard variation as well as in terms of efficiency;  we have a speed up of  a factor $2$ for $u=20000$ to a factor of $5$ for $u=500000$. \\
Summarizing, our numerical findings show that the novel MAK estimator proposed in this paper is by far the best from the considered ones. Since the efficiency of MAK in the above examples \cDD{is} at least a factor $2$ better than for the other estimators,
which  means that the evaluation time of $\alpha(u)$  for a given precision is at most half as long as for the other estimators,
\zH{our estimator shows clear advantages for practical applications.}
\begin{table}
\centering
\begin{tabular}{|c||c|c|c|c|c|}
 \hline
Method &Estimation &Standard error &Variation coeff. &Time& Efficiency\\\hline
RN & $0.00102$ & $0.000914$ & $0.892$ & $67.6$ & $34.1$\\\hline
MAK & $0.00102$ & $2.81e-05$ & $0.0275$ & $40.1$ & $60700$\\\hline
IS & $0.00102$ & $0.0166$ & $16.2$ & $2.82$ & $2.49$\\\hline
IS-CE & $0.00344$ & $7.41$ & $2150$ & $13$ & $2.7e-06$\\\hline
ISVE & $0.00102$ & $0.000472$ & $0.461$ & $14$ & $620$\\\hline
ISVE-CE & $0.00102$ & $0.00024$ & $0.235$ & $14.5$ & $2300$\\\hline
CMC & $0.00104$ & $0.0322$ & $31$ & $1.86$ & $1$\\\hline
\end{tabular}

\caption{ $\rho=0$, $u=20000$}
\end{table}
\notindraft{
\begin{table}
\centering

\begin{tabular}{|c||c|c|c|c|c|}
 \hline
Method &Estimation &Standard error &Variation coeff.&Time &Efficiency\\\hline
RN & $0.000463$ & $0.000415$ & $0.897$ & $66$ & $75.7$\\\hline
MAK & $0.000463$ & $9.11e-06$ & $0.0197$ & $39.5$ & $264000$\\\hline
IS & $0.000465$ & $0.00963$ & $20.7$ & $2.79$ & $3.33$\\\hline
IS-CE & $0.000415$ & $0.013$ & $31.4$ & $12.9$ & $0.396$\\\hline
ISVE & $0.000463$ & $0.000197$ & $0.425$ & $14$ & $1590$\\\hline
ISVE-CE & $0.000463$ & $0.00023$ & $0.496$ & $14.6$ & $1120$\\\hline
CMC & $0.000464$ & $0.0215$ & $46.4$ & $1.86$ & $1$\\\hline
\end{tabular}
\caption{ $\rho=0$, $u=40000$\label{tab02}}
\end{table}

 \begin{table}
\centering
\begin{tabular}{|c||c|c|c|c|c|}
 \hline
Method &Estimation &Standard error &Variation coeff.&Time &Efficiency\\\hline
RN & $1.79e-05$ & $1.82e-05$ & $1.01$ & $63$ & $1560$\\\hline
MAK & $1.8e-05$ & $7.93e-08$ & $0.00442$ & $39.5$ & $1.31e+08$\\\hline
IS & $1.83e-05$ & $0.000879$ & $48.1$ & $2.69$ & $15.6$\\\hline
IS-CE & $1.67e-05$ & $3.62e-05$ & $2.17$ & $12.6$ & $1960$\\\hline
ISVE & $1.79e-05$ & $1.57e-06$ & $0.0872$ & $14$ & $946000$\\\hline
ISVE-CE & $1.79e-05$ & $2.27e-07$ & $0.0127$ & $14.7$ & $42800000$\\\hline
CMC & $1.75e-05$ & $0.00418$ & $239$ & $1.85$ & $1$\\\hline
\end{tabular}
 \caption{ $\rho=0$, $u=500 000$\label{tab03}}
\end{table}

\begin{table}
\centering
\begin{tabular}{|c||c|c|c|c|c|}
 \hline

Method &Estimation &Standard error &Variation coeff.&Time &Efficiency\\\hline
RN & $0.00105$ & $0.000954$ & $0.908$ & $70.2$ & $32.2$\\\hline
MAK & $0.00105$ & $0.000146$ & $0.139$ & $46.1$ & $2090$\\\hline
IS & $0.00105$ & $0.0168$ & $16.1$ & $2.92$ & $2.48$\\\hline
IS-CE & $0.00104$ & $0.0236$ & $22.8$ & $12.8$ & $0.288$\\\hline
ISVE & $0.00105$ & $0.00271$ & $2.58$ & $14.1$ & $19.9$\\\hline
ISVE-CE & $0.00105$ & $0.00073$ & $0.695$ & $14.6$ & $265$\\\hline
CMC & $0.00105$ & $0.0324$ & $30.8$ & $1.96$ & $1$\\\hline
\end{tabular}

\caption{ $\rho=0.4$, $u=20000$}
\end{table}
\begin{table}
\centering

\begin{tabular}{|c||c|c|c|c|c|}
 \hline
Method &Estimation &Standard error &Variation coeff. &Time &Efficiency\\\hline
RN & $0.000473$ & $0.000428$ & $0.906$ & $69.9$ & $70.4$\\\hline
MAK & $0.000473$ & $5.66e-05$ & $0.12$ & $45.4$ & $6220$\\\hline
IS & $0.000468$ & $0.00956$ & $20.4$ & $2.87$ & $3.44$\\\hline
IS-CE & $0.000479$ & $0.0569$ & $119$ & $12.9$ & $0.0217$\\\hline
ISVE & $0.000472$ & $0.00123$ & $2.6$ & $14$ & $42.9$\\\hline
ISVE-CE & $0.000472$ & $0.000402$ & $0.85$ & $14.5$ & $385$\\\hline
CMC & $0.000471$ & $0.0217$ & $46.1$ & $1.92$ & $1$\\\hline
\end{tabular}
\caption{ $\rho=0.4$, $u=40000$}
\end{table}

 \begin{table}
\centering
\begin{tabular}{|c||c|c|c|c|c|}
 \hline
Method &Estimation &Standard error &Variation coeff. &Time &Efficiency\\\hline
RN & $1.81e-05$ & $1.83e-05$ & $1.01$ & $68.2$ & $1540$\\\hline
MAK & $1.81e-05$ & $1.15e-06$ & $0.0637$ & $42.6$ & $623000$\\\hline
IS & $1.8e-05$ & $0.000861$ & $47.9$ & $2.78$ & $17.1$\\\hline
IS-CE & $1.76e-05$ & $0.00107$ & $61.1$ & $13$ & $2.35$\\\hline
ISVE & $1.81e-05$ & $3.54e-05$ & $1.96$ & $14.1$ & $2000$\\\hline
ISVE-CE & $1.81e-05$ & $1.14e-05$ & $0.629$ & $14.8$ & $18500$\\\hline
CMC & $1.82e-05$ & $0.00427$ & $234$ & $1.94$ & $1$\\\hline
\end{tabular}
 \caption{ $\rho=0.4$, $u=500 000$}

\end{table}

\begin{table}
\centering
\begin{tabular}{|c||c|c|c|c|c|}
 \hline
Method &Estimation &Standard error &Variation coeff. &Time&Efficiency\\\hline
RN & $0.00113$ & $0.0012$ & $1.06$ & $84.6$ & $18.1$\\\hline
MAK & $0.00113$ & $0.000493$ & $0.437$ & $58.3$ & $155$\\\hline
IS & $0.00113$ & $0.0188$ & $16.6$ & $2.89$ & $2.15$\\\hline
IS-CE & $0.00113$ & $0.0022$ & $1.95$ & $13.1$ & $34.9$\\\hline
ISVE & $0.00113$ & $0.0096$ & $8.49$ & $14.1$ & $1.69$\\\hline
ISVE-CE & $0.00113$ & $0.00155$ & $1.38$ & $15$ & $60.7$\\\hline
CMC & $0.00112$ & $0.0335$ & $29.8$ & $1.96$ & $1$\\\hline
\end{tabular}

\caption{ $\rho=0.9$, $u=20000$}
\end{table}
\begin{table}
\centering

\begin{tabular}{|c||c|c|c|c|c|}
 \hline
Method &Estimation &Standard error &Variation coeff.  &Time&Efficiency\\\hline
RN & $0.000519$ & $0.000542$ & $1.04$ & $83.8$ & $41.1$\\\hline
MAK & $0.000519$ & $0.000215$ & $0.414$ & $57.6$ & $381$\\\hline
IS & $0.000515$ & $0.0108$ & $21.1$ & $2.85$ & $3.02$\\\hline
IS-CE & $0.000519$ & $0.00105$ & $2.02$ & $13$ & $71$\\\hline
ISVE & $0.000519$ & $0.00545$ & $10.5$ & $14.1$ & $2.41$\\\hline
ISVE-CE & $0.000519$ & $0.000706$ & $1.36$ & $14.9$ & $136$\\\hline
CMC & $0.000519$ & $0.0228$ & $43.9$ & $1.95$ & $1$\\\hline
\end{tabular}
\caption{ $\rho=0.9$, $u=40000$}
\end{table}

 \begin{table}
\centering
\begin{tabular}{|c||c|c|c|c|c|}
 \hline
Method &Estimation &Standard error &Variation coeff.&Time &Efficiency\\\hline
RN & $2.08e-05$ & $2.29e-05$ & $1.11$ & $81.8$ & $1010$\\\hline
MAK & $2.08e-05$ & $7.22e-06$ & $0.348$ & $54.8$ & $15200$\\\hline
IS & $1.95e-05$ & $0.000915$ & $46.8$ & $2.78$ & $18.7$\\\hline
IS-CE & $2.08e-05$ & $4.91e-05$ & $2.37$ & $12.7$ & $1420$\\\hline
ISVE & $2.1e-05$ & $0.000476$ & $22.7$ & $14$ & $13.7$\\\hline
ISVE-CE & $2.08e-05$ & $3.19e-05$ & $1.53$ & $14.7$ & $2900$\\\hline
CMC & $2.28e-05$ & $0.00477$ & $209$ & $1.91$ & $1$\\\hline
\end{tabular}
 \caption{ $\rho=0.9$, $u=500 000$}

\end{table}
}
\section{Proofs}
We prove next a lemma which is of independent interest, and then continue with the proofs of the main results.

\begin{lem}
\cE{ Let $\U$ be uniformly distributed on the unit sphere of $\R^d,d\ge2$ and $\Sigma=A A^\top$ be a correlation matrix ($\sigma_{ii}=1$ and $-1\le\sigma_{ij}=\sigma_{ji}\le1)$ with $A$ a lower triangular
non-singular matrix. Then the components of the random vector $\vk{\theta}=A \U$ satisfy for any $i\le d$}
\BQN\label{eq:comp}
|\Theta_i- \sigma_{i1} \Theta_1|=|\Theta_i- \sigma_{i1} U_1|\le\sqrt{1-\sigma_{i1}^2}\sqrt{1-\Theta_1^2}=\sqrt{1-\sigma_{i1}^2}\sqrt{1-U_1^2}.
\EQN
 \end{lem}
\begin{proof}
By the assumptions $\sum_{j=1}^d a_{ij}^2=\sigma_{ii}=1, i\le d$ and
$$
\bH{\Theta_1}= U_1, \quad \bH{\Theta_i}-\sigma_{i1} U_1=\sum_{j=2}^d a_{ji} \cDD{U_j}. 
$$
Since $\sum_{j=2}^d a_{ij}^2=1- \sigma_{i1}^2$ \cDD{and $\sum_{j=2}^d U^2_j=1- U_1^2$} the claim follows by Cauchy-Schwarz inequality.
\end{proof}
\cD{
\begin{korr}\label{cor:last} Under Assumption \aH{\eqref{beta:ey}} we have that for every $j$ with $\beta_j=\beta_1$ and every $\epsilon>0$ there exist \zH{$c,u_0$ positive such that}
\BQN
 \theta_i\le \theta_j \frac{\beta_j} {\beta_i} \frac{\log(\epsilon e^*_j(u))}{\log(u)}
\EQN
\zH{holds for all $\theta_j>1-c/\log(u)$.}
\end{korr}
}
\cD{
\begin{proof} Condition \bH{\eqref{beta:ey}} and  \eqref{eq:comp} imply
 $$
\theta_i\le \theta_j \left(\sigma_{ij}(u)+\sqrt{1-\sigma_{ij}(u)^2} \sqrt{\frac 1 {\theta_j^2}-1}\right) \le \theta_j \frac{\beta_j} {\beta_i} \frac{\log(\epsilon e^*_j(u))}{\log(u)},
$$
and hence the claim follows.
\end{proof}
}

\prooftheo{theo:dom:AK}
In view of Lemma \ref{theorem:Stratification} \eH{(in Appendix)} we have to analyze the estimators $\bH{Z_j(u)}$. \eH{For simplicity we assume that} $j=1$.
Next, \dH{suppose} that $\beta_1=\max_{1\le i\le d}\beta_i$. We have to show that in this case
\begin{equation}
\lim_{u\to\infty} \frac{\mean{Z_{1}(u)^2}}{\mean{Z_{1}(u)}^2}\le1\label{equ:toshowAKbmax}.
\end{equation}
For a  \bH{constant $c$ such that} $$
c>\sqrt{4\frac{\log(d)}{(\beta_1\gamma_u)^2}}
$$
we split the mean into two cases: $\max_{2\le i\le d} N_i>c\sqrt{\log(u)}$ and  $\max_{2\le i\le d} N_i\le c\sqrt{\log(u)}$.
For the first case note that
\begin{align*}
 \mean{ Z_{1}(u)^2 \mathbb{I}_{\left\{\max_{2\le i\le d} N_i>c\sqrt{\log(u)} \right\}} }
 &\le \sum_{i=2}^d \mean{Z_{1}(u)^2 \cH{\mathbb{I}}_{\left\{ N_i>c\sqrt{\log(u)} \right\}}}\\
&\le \sum_{i=2}^d \mean{\PP\left(\lambda_1e^{\beta_1\gamma_u N_1}>\frac u d\right)^2 \cH{\mathbb{I}}_{\left\{ N_i>c\sqrt{\log(u)} \right\}}}\\
&=d \ \PP\left(\lambda_1 e^{\beta_1\gamma_u N_1}>\frac u d\right)^2\PP\left( \zH{N_1}>c\sqrt{\log(u)}\right)\\
&\approx \exp\left(-\Biggl(\frac{\log\left(\frac{u}{d\lambda_1 }\right)}{\beta_1\gamma_u}\Bigr)^2-\frac{c^2\log(u)}2\right)\\
&\approx\exp\left(-\Biggl(\frac{\log\left(\frac{u}{\lambda_1 }\right)}{\beta_1\gamma_u}\Bigr)^2-\left(\frac{c^2}2-2\frac{\log(d)}{(\beta_1\gamma_u)^2} \right) \log(u) \right)\\
&=o\left(\PP\left(\lambda_1 e^{\beta_1\gamma_u N_1}>u \right)^2\right),
\end{align*}
where $\approx$ is \hH{a} logarithmic asymptotic \cK{and also the last equality holds on this logarithmic scale}.\\
For the second case we have that (\bH{with} $c_1>0$ is a suitable constant)
\aK{\begin{align*}
 &\mean{ Z_{1}(u)^2 \cH{\mathbb{I}}_{\left\{\max_{2\le i\le d} N_i\le c\sqrt{\log(u)} \right\}} }\\
&\le\mean{ \PP\left( \lambda_1 e^{\beta_1\gamma_u N_1} +\sum_{i=2}^d \lambda_i e^{\beta_i\gamma_u a_{i1}N_1+\sum_{j=2}^{i} \beta_i\gamma_u a_{ij}\cD{N}_j }>u,\lambda_1 e^{\beta_1\gamma_u N_1}>\frac ud \right)^2
\cH{\mathbb{I}}_{\left\{\max_{2\le i\le d} N_i\le c\sqrt{\log(u)}\right\} }}\\
&\le\mean{ \PP\left( \lambda_1 e^{\beta_1\gamma_u N_1} +\sum_{i=2}^d \lambda_i  e^{\beta_i\gamma_u a_{1i}N_1+\sum_{j=2}^{i} \beta_i\gamma_u \sqrt{1-a_{1i}^2} c\sqrt{\log(u)} }>u,\lambda_1 e^{\beta_1\gamma_u N_1}>\frac ud \right)^2
 }\\
&\le\mean{ \PP\left( \lambda_1 e^{\beta_1\gamma_u N_1} +\sum_{i=2}^d \lambda_i  e^{\beta_i\gamma_u a_{1i}N_1+d \beta_i\gamma_u \sqrt{1-a_{1i}^2} c\sqrt{\log(u)} }>u,\lambda_1 e^{\beta_1\gamma_u N_1}>\frac ud \right)^2
 }\\
 &=\PP\left( \lambda_1  e^{\beta_1\gamma_u N_1} +\sum_{i=2}^d \lambda_i  e^{\beta_i\gamma_uN_1\left( a_{1i}+d c \sqrt{1-a_{1i}^2} \sqrt{\frac{\log(u)}{N_1^2} }\right) }>u ,\lambda_1 e^{\beta_1\gamma_u N_1}>\frac ud\right)^2\\
  &\le\PP\left( \lambda_1  e^{\beta_1\gamma_u N_1} +\sum_{i=2}^d \lambda_i  e^{\beta_i\gamma_uN_1\left( a_{1i}+d c \sqrt{1-a_{1i}^2} \sqrt{\frac{(\beta_1\gamma_u)^2\log(u)}{\log(u/(d\lambda_1))^2}  } \right)  }>u \right)^2\\
&\lesssim\PP\left( \lambda_1  e^{\beta_1\gamma_u N_1} +\sum_{i=2}^d \lambda_i  e^{\beta_i\gamma_uN_1\left( a_{1i}+c_1 \sqrt{\frac{1-a_{1i}^2}{\log(u)}} \right) }>u \right)^2\\
&\le\PP\left( \lambda_1  e^{\beta_1\gamma_u N_1} +\sum_{i=2}^d \lambda_i  e^{\beta_1\gamma_u N_1  \frac {\log(\epsilon e^*_1(u))}{\log(u)} }>u \right)^2,
\end{align*}}
\bH{where we write $a_{ij}$ instead of $a_{ij}(u)$.} \cD{
\eH{Next, we} can find \bH{another constant}  $c_2$ such that for every $\epsilon$ there exists a $u_\epsilon>1$}
$$c_2>\sup_{u>u_{\cD{\epsilon}}} \sum_{i=2}^d\lambda_i  e^{\left(\log\left(1-c_2\epsilon\frac{e_1^*(u)}{u}\right)-\log(\lambda_1)\right) \frac {\log(\epsilon e^*_1(u))}{\log(u)} }.$$
 \cD{\dH{I}f we set
 $$\eH{L_1}=\frac{\log\left(\frac{u-c_2 \epsilon e^*_1(u)}{\lambda_1}\right)}{\beta_1\gamma_u}$$} then \cD{for $u>u_\epsilon$}
\begin{align*}
& \lambda_1  e^{\beta_1\gamma_u \eH{L_1}} +\sum_{i=2}^d \lambda_i  e^{\beta_1\gamma_u \eH{L_1}  \frac {\log(\epsilon e^*_1(u))}{\log(u)} }\\
&=u-c_2 \epsilon e^*_1(u)+ \epsilon e^*_1(u) \sum_{i=2}^d\lambda_i  e^{\left(\log\left(1-c_2\epsilon\frac{e_1^*(u)}{u}\right)-\log(\lambda_1)\right) \frac {\log(\epsilon e^*_1(u))}{\log(u)} }\le u.
\end{align*}
Hence \eqref{equ:toshowAKbmax} follows from
$$
\mean{ [Z_{1}(u)]^2 \cH{\mathbb{I}}_{\left\{\max_{2\le i\le d} N_i\le c\sqrt{\log(u)} \right\}} }\lesssim \PP\left( \lambda_1  e^{\beta_1\gamma_u N_1}>u-c_2\epsilon e_j^*(u) \right)^2
$$
and letting $\epsilon\to0$. On the other hand
 if $\beta_1 \cK{<} \max_{ i\le d } \beta_i$, then for all $u$ large
$$
Z_{1}(u) \le \pk{\lambda_1  e^{\beta_1\gamma_u N_1}>\frac ud }
$$
implying
$$
\limsup_{u\to\infty}\frac{ \left(\mean{Z_{1}(u)^2}\right)}{ \Psi_1(u) \pk{S(u)>u}}=0,
$$
and hence the \bH{claim} follows from  Lemma \ref{theorem:Stratification}.\QED\\

\def\vcTT{\theta}

\prooftheo{mytheo3}  \cK{Define $ \zH{\vk{\Theta}}_u:=A_u \bH{\vk{U}}$} and write $\bH{\Theta_i}$ for \zH{its} $i$th component. \bH{Note that by the assumption $\Theta_i$ has distribution function not depending on $u$.} Condition \bH{\eqref{eq:main:e} implies}  $\exp(R) \in GMDA(e)$ with $e(u)=u \nu(\log u)$. By the Davis-Resnick tail property (see e.g., Hashorva (2012) or Hashorva (2013b)) for any $c>1$ and $\mu>0$
\BQN\label{DRP}
\limit{u} \frac{ \PP\left(R > \cE{\log(cu)}\right)}{ \left(\frac{e(u)}u\right)^\mu \PP\left(R > \cE{\log(u)}\right)}&=& 0.
\EQN
In order to show the proof we use the next equality, that holds for all $u>d \lambda_j$
\BQNY
\PP\left(\SSU>u\right)&=&\sum_{j=1}^d \PP\left(\SSU>u,\Zju \bH{=\max_{k \le d} } \Zku \right)\\
&=&\sum_{\bH{j}=1}^d  \int_{0}^1  \PP\left(\sum_{i=1}^d \lambda_i e^{R \bH{\Theta_i} \beta_i \gamma_u}>u,\lambda_j  e^{R \bH{\Theta_j} \beta_j \gamma_u} \bH{=\max_{k \le d}}
\lambda_k  e^{R \bH{\Theta_k}\beta_k \gamma_u} \Bigl | \bH{\Theta}_j=\theta \right) \ftj(\theta) d \theta,
\EQNY
where without loss of generality we will assume  that depending on $j$ an $A_u$ is chosen such that  $\bH{\Theta_j}=U_j$ and  \dH{the pdf} $f$ is given by
\eqref{dtheta}.
For a fixed $j$,  we  split the integral above into two parts determined through $a(u)=1-2\log(d)/\log(u), u>d\lambda_j$.
\cD{Then we have that}
\begin{align*}
&\int_{0}^{a(u)}  \PP\left(\sum_{i=1}^d \lambda_i e^{R \bH{\Theta_i}  \beta_i \gamma_u}>u,\lambda_j  e^{R \Theta_j\beta_j \gamma_u}>\max_{k\not=j}\lambda_k  e^{R \Theta_k\beta_k \gamma_u}  \Bigl |\bH{\Theta_j} =\theta \right) \ftj(\theta) d \theta\\
&\le\int_{0}^{a(u)} \PP\left(d \lambda_j e^{R \theta \beta_j \gamma_u}>u\right)\ftj(\theta) d \theta 
\\&=o\left(\PP\left( \Zju >u\right)\right).
\end{align*}
The last equality follows as a combination of \zH{\eqref{DRP} and } Lemmas \ref{remarkunivariateasymptotic}, \cD{\ref{remarkunivariateasymptotic2}}
 \dH{in Appendix}. \\
\cE{Further for any $\epsilon\in (0,1)$  and $u>u_0$ ($u_0$ from condition \ref{beta:ey}) we obtain} \cD{by Corollary \ref{cor:last}}
\begin{align*}
&\int_{a(u)}^1  \PP\left(e^{R \bH{\Theta_i}   \beta_i \gamma_u}>u \Bigl |\bH{\Theta_j} =\theta \right) \ftj(\theta) d \theta\\
&\cD{\lesssim} \int_{a(u)}^1  \PP\left(\sum_{i=1}^d \lambda_i e^{R \Theta_i  \beta_i \gamma_u}>u,\lambda_j  e^{R \bH{\theta} \beta_j \gamma_u}>\max_{k\not=j}\lambda_k  e^{R \Theta_k\beta_k \gamma_u}|\bH{\Theta_j} =\theta \right) \ftj(\theta) d \theta\\
&\le \int_{a(u)}^1 \PP\left( \lambda_j e^{R \theta  \beta_j \gamma_u}+\sum_{i\not=j} \lambda_i e^{R \theta \beta_j \gamma_u \frac {\log(\epsilon e^*_j(u))}{\log(u)}  }>u\right) \ftj(\theta) d \theta.
\end{align*}
\cD{Next, we choose $c$ such that for all $\tau<1$ there {exists} a  $u_\tau$ such that}
$$
c> \sup_{u>u_{\cD{\tau}}} \sum_{i\not= j} \lambda_i  e^{\left(\log\left(1-\frac{c \tau e_j^*(u)} u\right) - \log(\lambda_j) \right) \frac {\log(\tau e^*_j(u))}{\log(u)}  }.
$$
Both constants  $u_{\cD{\tau}}$ and $c$ \zH{exist} since we assume that $\limit{u} e(u)=\IF$.  Hence for $u>u_{\bH{\tau}}$ 
\BQNY
 \int_{a(u)}^1 \PP\left( \lambda_j e^{R \theta  \beta_j \gamma_u}+\sum_{i\not=j} \lambda_i e^{R \theta \beta_j \gamma_u \frac {\log(\epsilon e^*_j(u))}{\log(u)}  }>u\right) \ftj(\theta) d \theta
&\le&\int_{a(u)}^1 \PP\left( \lambda_j e^{R \theta  \beta_j \gamma_u}>u-c\epsilon e_j^*(u)\right) \ftj(\theta) d \theta.
\EQNY
Assuming that $\beta_j=\max_{i \bH{\le d}}\beta_i$, Lemma \ref{remarkunivariateasymptotic}  implies thus 
$$\eH{\Psi_j(u)}=\PP\left(\SSU>u,\Zju>\max_{k\not=j}\Zku \right)\sim \PP\left(\Zju>u \right).$$
If $\beta_j<\max_{\bH{i\le d}}\beta_i$ \aK{and $k$ \dH{is} such that $\beta_k=\max_{\bH{i\le d}}\beta_i$}, then for every $c_1>1$
\BQNY
\eH{\Psi_j(u)} &\le & \PP\left( X_{k}(u)>u/d \right)\\
&= &\PP\left( X_k(u)>\left(\frac{u}{d\lambda_k}\right)^{\frac{\beta_{k}}{\beta_j}}\right)
\lesssim \PP\left(X_{k}(u)>c_1u\right),
\EQNY
hence the claim follows. \QED

\zH{\prooftheo{lemma:rojas1}} \aK{Again we have to analyze the \zH{estimator} $\hat Z_j(u)$. \hH{Denote by} $f_{-j}(\vk{\theta}_{-j}|\theta)$ the conditional density of $\vk{\Theta}_{-j}:=(\Theta_1,\ldots,\Theta_{j-1},\Theta_{j+1},\ldots,\Theta_d)$ given $\Theta_j=\theta$.
The second moment of the estimator is  given by
\begin{equation}
\int_{-1}^1 \int \PP\left(\left.S(u)>u,\Zju =\max_{k \le d }\Zku \right| \vk{\Theta}=\vk{\theta} \right)^2 f_{-j}(\vk{\theta}_{-j}|\theta)\,
 d \vk{\theta}_{-j}\frac{ f( \vcTT )} {f_{IS}( \vcTT  )} f( \vcTT )\, d{ \vcTT }.
\label{equ:second:estimator}
\end{equation}
We assume next that $\beta_j=\max_{1\le i\le d} \beta_j$. As in the proof of Theorem \ref{mytheo3} we split the integral  into parts, where $\Theta_j$ is between $a(u)$ and $1$ respectively  $-1$ and $a(u)$. By the same method as in the proof of Theorem \ref{mytheo3} for some $c>0$ and all $\epsilon \in (0,1)$ we obtain
\begin{align*}
&\int_{a(u)}^1 \PP\left(\left. \lambda_je^{R\bH{\Theta}_j\beta_j\gamma_u} >u \right|\bH{\Theta}_j=\vcTT\right)^2\frac{f( \vcTT)}{f_{IS}(\vcTT )} f(\vcTT)\, d\vcTT\\
&\lesssim \int_{a(u)}^1 \int \PP\left(\left.S(u)>u,\Zju>\max_{k\not=j}\Zku \right|{ \vk{\Theta}=\vk{\theta} }\right)^2 f_{-j}(\vk{\theta}_{-j}|\theta) d \vk{\theta}_{-j} \frac{f(\vcTT )} {f_{IS}( \vcTT  )} f( \vcTT )\, d{ \vcTT }\\
&\lesssim \int_{a(u)}^1 \PP\left(\left. \lambda_je^{R\Theta_j\beta_j\gamma_u} >u-c\epsilon e_j^*(u) \right|\Theta_j=\vcTT\right)^2\frac{f(
\vcTT)}{f_{IS}(\vcTT )} f(\vcTT)\, d{\vcTT}.
\end{align*}}
As in the proof of Lemma \ref{remarkunivariateasymptotic} we substitute $\vcTT=\frac{\log(u)}{\log(u)+\log(1+xe^*_j(u)/u) }$.
\zH{Set next $\beta(u):=\bH{\log(u)}/\log(\bH{u/\dH{e^*}(u)})$ and} note that uniformly for $u\to \infty$
\begin{align*}
&\frac{f\left(\frac{\log(u)}{\log(u)+\log(1+xe^*_j(u)/u) }\right)}
{B(\bH{a} ,\beta(u))f_{IS}\left( \frac{\log(u)}{\log(u)+\log(1+xe^*_j(u)/u) }\right)}\\&\sim 2^{\frac{d-3}2+\beta(u)} \frac{\Gamma(d/2)}{\sqrt{\pi} \Gamma((d-1)/2)}\left(\frac{\log(1+xe^*_j(u)/u)}{\log(u)+\log(1+xe^*_j(u)/u) }\right)^{\frac{d-3}2-\beta(u)+1}\\
&\le e 2^{\frac{d-3}2+\beta(u)} \frac{\Gamma(d/2)}{\sqrt{\pi} \Gamma((d-1)/2)}\left(\frac{xe^*_j(u)}{u\log(u) }\right)^{\frac{d-1}2} x^{-\beta(u)}
\end{align*}
and
$$
\frac{f\left(\frac{\log(u)}{\log(u)+\log(1+xe^*_j(u)/u) }\right)}
{f_{IS}\left( \frac{\log(u)}{\log(u)+\log(1+xe^*_j(u)/u) }\right)}\sim e 2^{\frac{d-3}2} \frac{B(\bH{a},\beta(u)) \Gamma(d/2)}{\sqrt{\pi} \Gamma((d-1)/2)}\left(\frac{e^*_j(u)}{u\log(u) }\right)^{\frac{d-1}2} x^{\frac{d-1}2},
$$
where $\bH{B(a ,\beta(u))}= \Gamma(a)\Gamma(\beta(u))/ \Gamma(a+\beta(u))$. Consequently, as in the proof of Lemma \ref{remarkunivariateasymptotic} \begin{align}
& \int_{a(u)}^1 \PP\left(\left. \lambda_je^{R\Theta_j\beta_j\gamma_u} >u-c\epsilon e_j^*(u) \right|\bH{\Theta_j= \vcTT}\right)^2\frac{f(
\vcTT)}{f_{IS}(\vcTT )} f(\vcTT) d{\vcTT}\notag\\
&\sim (1+O(\epsilon))\left(\frac{ 2^{\frac{d-3}2} \Gamma(d/2)}{\sqrt{\pi} \Gamma((d-1)/2)}\left(\frac{e^*_j(u)}{u\log(u)}\right)^{\frac{d-1} 2}\PP\left( \lambda_j   e^{R  \beta_j \gamma_ {u}}>u \right)\right)^2\notag\\
&\quad\times eB(\bH{a},\beta(u)) \int_{0}^\infty e^{-2x} x^{\frac{d-3}2+\frac{d-1}2}  d x\notag\\
&\sim- \cK{ \frac{e \Gamma(d/2) } {2\sqrt{\pi}\Gamma((d-1)/2) }}\log \left( \frac{e^*_j(u)}{u\log(u)}
\right)
 \PP\left(  X_j(u)>u \right)^2.
 \label{asymrnimax}
\end{align}
\aK{Since $B(\bH{a},\beta)\sim 1/\beta$ as $\beta\to 0$, \zH{analogously to the proof of Lemma \ref{remarkunivariateasymptotic2}} we have
\begin{align*}
&\int_{-1}^{a(u)} \int\PP\left(\left.S(u)>u,\Zju>\max_{k\not=j}\Zku \right| \vk{\Theta}=\vk{\theta} \right)^2 f_{-j}(\vk{\theta}_{-j}|\theta) d \vk{\theta}_{-j}\frac{f(\vcTT)}{f_{IS}( \vcTT  )} f( \vcTT ) d{ \vcTT }\\
&\le \int_{-1}^{a(u)} \PP\left(\left.\Zju>u/d \right|{\Theta_j=\vcTT   }\right)^2\frac{f(\vcTT)}{f_{IS}( \vcTT  )} f( \vcTT ) d{ \vcTT }\\
&\sim \int_{0}^{a(u)} \PP\left(\left.\Zju>u/d \right|{\Theta_j=\vcTT   }\right)^2\frac{f(\vcTT)}{f_{IS}( \vcTT  )} f( \vcTT ) d{ \vcTT }\\
&=\log\left(\frac{u\log(u) }{\bH{e^*_j(u)}}\right) o\left( \PP(X_j(u)>u)^2\right).
\end{align*}}
Next assume that $\beta_j<\max_{1\le  i \le d} \beta_i$. The second moment of the estimator is (asymptotically) given by \eqref{equ:second:estimator}. As in the proof of \dH{Theorem} \ref{mytheo3} for every $c>1$ we obtain
\aK{\begin{align*}
&\int_{-1}^{1}\int \PP\left(\left.S(u)>u,\Zju>\max_{k\not=j}\Zku \right| \vk{\Theta}=\vk{\theta} \right)^2 f_{-j}(\vk{\theta}_{-j}|\theta) d \vk{\theta}_{-j}\frac{f(\vcTT)}{f_{IS}(\vcTT )} f({\bf  \theta}) d{ \vcTT }\\
&\lesssim \int_{0}^{1} \PP\left(\left.d \Zju>u \right|\Theta_j=\vcTT\right)^2\frac{f(\vcTT)}{f_{IS}(\vcTT )} f(\vcTT) d{\vcTT}.
\end{align*}}
\dH{We} can proceed as in the proof of \zH{Theorem} \ref{lemma:rojas1} to get that
$$
\mean{[\hat Z_{j}\dH{(u)]^2}} =o \left(\log \left( \frac{u\log(u)} {e^*_j(u)}\right)\left(\frac{ e\left(\left(\frac{u}{d\lambda_j}\right)^{\frac1{\beta_j\gamma_u}}\right)}{e\left(\left(\frac{u}{\lambda_j}\right)^{\frac1{\beta_j\gamma_u}}\right)}\right)^{\frac{d-1} 2} \right) \PP\left(  X_j(u)>u/d \right)^2
$$
and hence the claim  follows by condition \eqref{bed:rojas} and Lemma \ref{theorem:Stratification}. \QED
\appendix

\section{Appendix}
In the sequel we consider some positive random variable  $R$ 
such that its distribution function  $F$
has an infinite upper endpoint. We have the following representation for $F \in GMDA(\nu)$, see e.g., Resn\cD{ic}k (1987) 
\BQN\label{rep}
1- F(u)& =& c(u) \exp \Bigl( - \int_{x_0}^u \frac{g(t)}{\nu(t)}\, dt \Bigr),
\EQN
with $x_0$ some constant and $c,g$ two positive measurable functions such that $\limit{u}c(u)=\limit{u}g(u)=1$.

Further we assume that $e(u)= u\cH{\nu(\log(u))}$ is a scaling function of $\exp(R)$, i.e., $\exp(R)\in GMDA(e)$.
This holds in particular when $\limit{u} \nu(u)\cH{=0}$.  We define $e^*(u)$ by \eqref{eq:eyj} for some $\lambda, \beta$ positive, \hH{i.e.,}
\hH{$$
 e^*(u)= \beta \gamma_u u e\left( \luiu^ {\frac 1 {\beta\gamma_u}}\right)\luiu^ {-\frac 1 {\beta\gamma_u}},
 $$
with $\gamma_u$ such that $\limit{u} \gamma_u= \gamma\in (0,\IF)$.
}
\hH{We proceed with two lemmas and then conclude} this section with two results, the first  shows an unbiased estimator for sums of certain probabilities, whereas the second provides an upper bound on the linear combination of the components of uniformly distributed random vectors on the unit sphere of $\R^d$.

\begin{lem}\label{remarkunivariateasymptotic}
\hH{Let $R$ be a positive random variable}, and let $\ftj$ be \hH{the pdf} given by \eqref{dtheta}. \hH{If $\exp(R)\in GMDA(e)$}, then for any
\hH{$\beta,\lambda,m,\ve$ positive} and some $k>0$
 \begin{equation}
\int_{a(u)}^1\PP\left( \lambda e^{R \theta \beta \gamma_ {u}}>u-\epsilon e^*(u)\right)^m \ftj(\theta) d \theta\\
\cK{=}  \cD{m^{\frac{d-1}2}}(1+O(\epsilon)) \frac{2^{\cD{\frac{d-\cK{3}}2}}\Gamma(d/2)}{\sqrt{\pi}}\left(\frac{e^*(u)}{u\log(u)}\right)^{\frac{d-1} 2}\PP\left( \lambda e^{R \beta \gamma_ {u}}>u \right)^{\cD{m}},
 \end{equation}
\cE{with $a(u)\le 1-k/\log(u)$} \hH{such that} $ \lim_{u\to\infty} a(u)=1$,
 \zH{and $\gamma_u$ \dH{some} positive constants such that $\limit{u} \gamma_u=\gamma\in (0,\IF)$}.
\end{lem}

\def\vuu{\bH{\xi(u)}}
\begin{proof} \hH{The assumption that $\exp(R) \in GMDA(e)$ implies}
\hH{
\BQN
\vuu :=e^*(u)/u \to 0, \quad u\to \IF.
\EQN
}
\hH{Next, set} \cE{$b(x,u)=\log(1+x \vuu )$} and $c=\frac{2^{\frac{d-\cK{3}}2}\Gamma(d/2)}{\sqrt{\pi} \Gamma((d-1)/2)}$. We have
\begin{align*}
&\int_{a(u)}^1\PP\left( \lambda e^{R \theta \beta \gamma_ {u}}>u-\epsilon e^*(u)\right)^{\cD{m}} \ftj(\theta) d \theta\\
&\sim c\int_{a(u)}^1\PP\left( \lambda e^{R \theta  \beta \gamma_ {u}}>u-\epsilon e^*(u)\right)^{\cD{m}} (1-\theta)^{\frac{d-3}2} d \theta \\
&=c\int_{0}^{\frac{ (u-\epsilon e^*(u))^{ 1/a(u)-1}-1}{ \vuu }} \frac{ \vuu \log(u-\epsilon e^* (u))}{\left(x  \vuu   + 1 \right)\left( \log(u-\epsilon e^* (u)) +b(x,u)\right)^2 } \\&\quad\quad\times\PP\left( \lambda  e^{R  \beta \gamma_ {u}}>\left(u-\epsilon e^*(u)\right)^{{1+\frac{\log(1+x  \vuu ) }{\log(u-\epsilon e^* (u))} }}\right)^{\cD{m}}  \left(\frac{ b(u,x)}{\log(u-\epsilon e^* (u))+b(u,x)}\right)^{\frac{d-3}2} dx\\
&\sim c \left(\frac{ \vuu }{\log(u)}\right)^{\frac{d-1} 2} \int_{0}^{\frac{(u-\epsilon e^*(u))^{ 1/a(u)-1}-1}{ \vuu }} \frac{1}{1+x \vuu } \left(1+\frac{b(u,x)}{ \log(u-\epsilon e^* (u)) }\right)^{-\frac {d+1} 2}\\
&\quad\quad\times\PP\left( \lambda   e^{R  \beta \gamma_ {u}}>\left(u-\epsilon e^*(u)\right)\left(1+x \cE{ \vuu }\right) \right)^{\cD{m}}
\left(\frac{ b(u,x)}{ \vuu  }\right)^{\frac{d-3}2} dx.
\end{align*}
\hH{It follows that $R \in GMDA(\nu)$, where $\nu(\log(u))= e(u)/u$,} hence  Eq. (6.31) of Hashorva (2009) implies for any $\ve>0$ and some $\eta_1,\eta_2$ positive constants
\begin{equation}
\frac{\PP(R> u +x \nu (u))}{\PP(R>u)}\le \eta_1 (1+ \eta_2 x)^{-1/\ve}.\label{test}
\end{equation}
Consequently, by the dominated convergence theorem
\begin{align*}
&\int_{0}^{\frac{ u^{ 1/a(u)-1}-1}{\cE{ \vuu }}} \frac{ \vuu \log(u)}{\left(x  \vuu   + 1 \right)\left( \log(u) +b(u,x)\right)^2 } \\&\quad\quad\times\PP\left( \lambda  e^{R  \beta \gamma_ {u}}>u+x(1+O(\epsilon))e^*(u) \right)^{\cD{m}}  \left(\frac{ b(u,x)}{\log(u)+b(u,x)}\right)^{\frac{d-3}2} dx\\
&\sim \left(\frac{ \vuu }{\log(u)}\right)^{\frac{d-1} 2}\PP\left( \lambda e^{R  \beta \gamma_ {u}}>u \right)^{\cD{m}}\int_{0}^\infty e^{-m x(1+O(\epsilon))} x^{\frac{d-3}2}d x\\
&\sim \cD{m^{\frac{d-1}2}} (1+O(\epsilon)) \frac{2^{\frac{d-\cK{3}}2}\Gamma(d/2)}{\sqrt{\pi}}\left(\frac{ \vuu }{\log(u)}\right)^{\frac{d-1} 2}\PP\left( \lambda  e^{R  \beta
 \gamma_ {u}}>u \right)^{\cD{m}}, 
\end{align*}
and thus the proof is complete. \end{proof}

\begin{lem}\label{remarkunivariateasymptotic2}
\hH{Under the assumptions of Lemma \ref{remarkunivariateasymptotic}, for any $\beta,\lambda,\ve$ positive} and some $k>d$
 \begin{equation}
\int_0^{1-\log(k)/\log(u)}\PP\left( \lambda e^{R \theta \beta \gamma_ {u}}>u/d\right) \ftj(\theta) d \theta\\
\ll   \frac{2\Gamma(d/2)}{\sqrt{\pi}}\left(\frac{e^*(u)}{u\log(u)}\right)^{\frac{d-1} 2}\PP\left( \lambda e^{R \beta \gamma_ {u}}>u \right),
 \end{equation}
where for two functions  $h_1(u)\ll h_2(u)$ means $h_1(u)=o(h_2(u))$.
\end{lem}

\begin{proof}
Choose  $b(u)\le a(u):=1- log(k)/\log(u)$ with $\lim_{u\to \IF}b(u)= 1$, such that $$\PP\left( \lambda e^{R b(u) \beta \gamma_ {u}}>u/d\right)=o\left(\int_0^1\PP\left( \lambda e^{R \theta \beta \gamma_ {u}}>u\right) \ftj(\theta) d \theta \right).$$
Set $ \vuu =e^*(u)/u, b(x,u)=\log(1+x \vuu )$. By substituting
$$
\theta=\frac{\log(u/k)}{\log(u+x e^*(u))}
$$
and for some $c>0$, we have that
\BQNY
\int_{b(u)}^{a(u)}\PP\left( \lambda e^{R \theta \beta \gamma_ {u}}>u/d\right) \ftj(\theta) d \theta
&\sim & c  \int_{0}^{\frac{ (u/k)^{ 1/b(u)}-u}{e^*(u)}} \frac{e^*(u) \log(u/k)} { (u+xe^*(u)) \log(u+x e^* (u))^2  } \\
&\times& \PP\left( \lambda  e^{R  \beta \gamma_ {u}}>\left(u/d\right)^{{\frac{\log(u+x e^*(u)) }{\log(u/k)} }}\right)  \left(\frac{ b(u,x)+\log(k)}{\log(u)+b(u,x)}\right)^{\frac{d-3}2} dx\\
&\lesssim & c \log(u)^{-\frac{d-1}2} \int_{0}^{\frac{ (u/k)^{ 1/b(u)}-u}{e^*(u)}} \frac{e^*(u)} { (u+xe^*(u))   } \\
&\times &  \PP\left( \lambda  e^{R  \beta \gamma_ {u}}>\frac kd \left(u+xe^*(u)\right)\right)  \left(\frac{ b(u,x)+\log(k)}{1+b(u,x)/\log(u)}\right)^{\frac{d-3}2} dx.
\EQNY
Next, \eqref{DRP} implies
\begin{align*}
\PP\left( \lambda  e^{R  \beta \gamma_ {u}}>\frac kd \left(u+xe^*(u)\right)\right) &\lesssim\left(
 \frac{e\left(\left(\frac{1} {\lambda }(u+xe^*(u))\right)^{\frac 1 {\beta \gamma_u }} \right)}{ \left(\frac{1} {\lambda }(u+xe^*(u))\right)^{\frac 1 {\beta \gamma_u }}} \right)^{\frac{d-1} 2}\PP\left( \lambda  e^{R  \beta \gamma_ {u}}> u+xe^*(u)\right)\\
\end{align*}
From the representation theorem for self-neglecting functions (cf. Bingham et al. (1987)) it follows that for every $\delta >0$ and u large enough  and $x>0$ \zH{we have}
$$e(u+xe(u))/e(u) \le(1+\delta) (1+\delta x).$$
Together with \eqref{test} \zH{it follows that} for every $\epsilon>0$ there \zH{exist} $\eta_1$ and $\eta_2$ such that
$$
\left(
 \frac{e\left(\left(\frac{1} {\lambda }(u+xe^*(u))\right)^{\frac 1 {\beta \gamma_u }} \right)}{ \left(\frac{1} {\lambda }(u+xe^*(u))\right)^{\frac 1 {\beta \gamma_u }}} \right)^{\frac{d-1} 2}\PP\left( \lambda  e^{R  \beta \gamma_ {u}}> u+xe^*(u)\right) \lesssim \eta_1(1+\eta_2 x)^{-1/\epsilon}  \left( \frac{e^*\left(u\right)}{u}\right)^{\frac{d-1} 2}\PP\left( \lambda  e^{R  \beta \gamma_ {u}}> u\right)
$$
\zH{holds uniformly for $x>0$}, and hence the proof follows with similar arguments as that of  Lemma \ref{remarkunivariateasymptotic}.
\end{proof}

\COM{
\begin{lem}\label{lemmaboundtailedu}If Assumptions \ref{Assumptions:firstorder} -\ref{cond4} is fulfilled then for every fixed $c>1$ and $k>0$
 \begin{equation*}
\PP\left(R > \cE{\log(cu)}\right)\lesssim \left(\frac{e(u)}u\right)^k \PP\left(R > \cE{\log(u)}\right).
 \end{equation*}
\end{lem}
\begin{proof} For some $c_1>0$
 \begin{align*}
 \frac{\PP\left(R >\log(cu)\right)}{\PP\left(R > \log(u)\right)} &\sim  \exp\left(-\int_{u}^{cu} \frac{1}{e(x)} dx\right)
= \exp\left(- \frac{u}{e(u)}\int_{1}^{c} \frac{e(u)}{e(ux)} dx\right)
\lesssim\exp\left(-\frac{u}{e(u)} c_1\right),
\end{align*}
thus the result follows. \end{proof}

\begin{lem}\label{lemmaboundtailu} If for some constant $c_0\in [0,\IF)$ we have
\BQNY
\limit{u} \frac{ \log(u) e(u)}{u}&=& c_0,
\EQNY
then for every fixed $c>1$ and $\epsilon>0$
 \begin{equation*}
\PP\left(\cE{R} >\log(cu)\right)\lesssim u^{-\frac{\log(c)}{c_0+\epsilon}}\PP\left( R >\log(u)\right).
 \end{equation*}
\end{lem}
\begin{proof}
\begin{align*}
 \frac{\PP\left(R > \log(cu)\right)}{\PP\left(R >\log(u)\right)} &\sim  \exp\left(-\int_{u}^{cu} \frac{1}{e(x)} dx\right)\\
&\lesssim  \exp\left(- \frac{1}{(c_0+\epsilon)} \int_{u}^{cu} \frac{\log(x)}{x} dx\right)\\
&=\exp\left(- \frac{1}{(c_0+\epsilon)} \int_{1}^{c} \frac{\log(u x)}{x} dx\right)\\
&\le\exp\left(-\frac{1}{(c_0+\epsilon)} \int_{1}^{c} \frac{\log(u )}{x} dx\right)=u^{-\frac{\log(c)}{c_0+\epsilon}}.
\end{align*}
\end{proof}
}

\begin{lem}\label{theorem:Stratification}
 Assume that that \cE{$A_i,i\le d$} are events and $Z_i$ is an unbiased estimator for $\PP(A_i)$ for $i\le d$. Let $\kal{I}$ be an integer valued random variable with
$$
\cE{\PP(\kal{I}=i)=\frac{z_i }{z}, \quad z:=\sum_{j=1}^d z_j,}
$$
with $z_i>0$ some positive constants. Then
$
Z:=z\sum_{k=1}^d \mathbb{I}_{\{\kal{I}=k\}} \frac {Z_k} {z_k }
$
 is an unbiased estimator for $\sum_{i=1}^d \PP(A_i)$ with
$$
\mean{Z^2}=z\sum_{i=1}^d \frac{\mean{Z_i^2}}{z_i}.
$$
\end{lem}
\begin{proof} The proof follows by straightforward calculations. \end{proof}

\textbf{Acknowledgments.} The authors would like to thank the referees for their careful reading and helpful comments.
\hH{Dominik Kortschak} has been  supported by the Swiss National Science Foundation Project 200021-124635/1.
\hH{Both authors also acknowledge partial support from Swiss National Science Foundation Project 200021-1401633/1.}

\end{document}